\documentclass[11pt,leqno]{article}
\usepackage[utf8]{inputenc}
\usepackage{amsmath,amssymb,amsthm,latexsym}
\usepackage{indentfirst}
\usepackage{amsmath}

\newtheorem{theorem}{\bf \large Theorem}[section]

\newtheorem{corollary}{\bf \large Corollary}[section]
\newtheorem{lemma}{\bf \large Lemma}[section]

\title{Rigidity of closed $\lambda$-self-expanders}

\author {{Tongzhu Li,~~~Mengdan Qi} \\
\small{Department of Mathematics, Beijing Institute of
Technology, Beijing, 100081, China.} \\
\small{ E-mail:litz@bit.edu.cn,~ mengdanqi@bit.edu.cn.}}
\date{}

\begin{document}
\maketitle
\begin{abstract}
A $\lambda$-self-expander $x: M^n\to \mathbb{R}^{n+1}$ is  the solution of the isoperimetric problem of weight $e^{\frac{|x|^2}{4}}$.
In this paper, we prove that  the closed  $\lambda$-self-expander is a round sphere centered at the origin under some  curvature conditions.  These curvature conditions mainly include the scalar curvature, the mean curvature, and the squared norm of
the second fundamental form.
\end{abstract}


\section{Introduction}

An immersed hypersurface $x: M^n \to \mathbb{R}^{n+1}$ is called a $\lambda$-self-expander if it satisfies
\begin{equation}\label{def}
    H - \frac12 \langle x,\xi \rangle = \lambda,
\end{equation}
where $\xi$ is  the inward unit normal vector field of $x$,  $H$ is the mean curvature of $x$, and $\lambda$ is a constant.
The concept of $\lambda$-self-expanders was first proposed by Ancari and Cheng in \cite{AncariCheng}. When $\lambda=0$, the equation reduces to the defining equation of standard self-expanders, so $\lambda$-self-expanders serve as a natural generalization of classical self-expanders. Moreover, $\lambda$-self-expanders are critical points of a weighted area functional under volume constraint, with the weight function given by $e^{\frac{|x|^2}{4}}$.
These hypersurfaces are dual to the well-studied $\lambda$-hypersurfaces, which satisfy $H+\tfrac12\langle x,n\rangle=\lambda$ and correspond to self-shrinking weighted solitons. Despite their natural variational origin, the global rigidity and classification theory for  closed $\lambda$-self-expanders remains significantly underdeveloped compared to the mature literature on $\lambda$-hypersurfaces, and few sharp spherical rigidity results have been established under natural curvature sign conditions.

In recent years, the rigidity theory of self-expanders has been extensively developed from various perspectives. The study of self-expanders begins with the lowest-dimensional case. Ishimura \cite{Ishimura} and Halldorsson \cite{Halldorsson} independently gave a complete classification of self-expander curves in $\mathbb{R}^2$. Ancari and Cheng \cite{AC} proved that the surfaces $\Gamma \times \mathbb{R}$ with the product metric are the only complete self-expander surfaces immersed in $\mathbb{R}^3$ with constant scalar curvature, where $\Gamma$ is a complete self-expander curve immersed in $\mathbb{R}^2$. And they also gave a complete classification of $n$-dimensional self-expanders in $\mathbb{R}^{n+1}$ with non-negative constant scalar curvature. In \cite{LSY}, Luo, Sun and Yin gave a slightly different proof of it. Furthermore, Cheng and Zhou \cite{CZ} studied the drifted Laplacian $L=\Delta+\frac12\langle x,\nabla\cdot\rangle$ on complete properly immersed self-expanders. Bernstein and Wang developed a comprehensive theory of asymptotically conical self-expanders in a series of papers \cite{BW2,BW1,BW3}. Smoczyk \cite{Smoczyk} studied mean convex self-expanding hypersurfaces and proved that a complete mean convex self-expander is a product of a self-expanding curve and a flat subspace if and only if the function $\frac{|h|^2}{|H|^2}$ attains a local maximum. As a corollary, complete mean convex self-expanders have strictly positive scalar curvature if they are smoothly asymptotic to cones of non-negative scalar curvature.

The constancy of the squared norm of the second fundamental form has also played an important role in rigidity and classification problems for self-similar hypersurfaces. For self-shrinkers, Cheng and Peng \cite{ChengPeng2015} studied complete self-shrinkers by means of a generalized maximum principle and obtained classification results under the assumption that the squared norm of the second fundamental form is constant. In dimension three, Cheng, Li and Wei \cite{ChengLiWei2022} classified complete self-shrinkers in $\mathbb{R}^{4}$ under the assumptions that $S$ and the fourth-order curvature invariant $f_4$ are constant. More recently, Fu and Yang \cite{FuYang2024} obtained further rigidity results for self-shrinkers with constant squared norm of the second fundamental form.

In the related theory of $\lambda$-hypersurfaces, Pengpeng Cheng and Tongzhu Li \cite{ChengLi2026} studied complete $\lambda$-hypersurfaces with constant squared norm of the second fundamental form and obtained classification results under the additional assumption that the third-order curvature invariant $f_3$ is constant. Their result includes, as a special case, a classification of complete self-shrinkers with constant $S$ and constant $f_3$. These works show that the condition $S=\mathrm{constant}$ is a natural and effective rigidity assumption in the study of weighted self-similar hypersurfaces.

For self-expanders, rigidity results under the assumption $S=\mathrm{constant}$ are comparatively more limited. Li and Wei \cite{LW} classified three-dimensional complete self-expanders in $\mathbb{R}^{4}$ under the assumptions that the squared norm $S$ of the second fundamental form and $f_3=\operatorname{tr}(h^3)$ are both constant. Subsequently, Li and Mi \cite{LiMi2025} obtained a classification of three-dimensional complete self-expanders in $\mathbb{R}^{4}$ assuming that $S$ and the fourth-order curvature invariant $f_4$ are constant. These results naturally lead to the question of whether the constancy of $S$ alone is sufficient to force rigidity when compactness is imposed.

Despite the comprehensive structural descriptions of classical self-expanders discussed above, the rigidity theory of $\lambda$-self-expanders is still relatively less developed. The notion of $\lambda$-self-expanders was introduced by Ancari and Cheng \cite{AncariCheng}, who established several fundamental rigidity and classification results. In particular, they obtained rigidity results for closed convex $\lambda$-self-expanders, classified complete $\lambda$-self-expanders with constant mean curvature, and derived rigidity criteria for closed $\lambda$-self-expanders under additional curvature assumptions involving the second fundamental form.

Further rigidity results have subsequently been obtained under different geometric assumptions. In dimension two, Deng, Yao, Yu and Zhang \cite{DengYu2024} obtained a local classification of $\lambda$-self-expanders with constant squared norm of the second fundamental form. This suggests considering whether the constancy of $S$ yields a global rigidity result for closed immersed $\lambda$-self-expanders in arbitrary dimension. More recently, Bai and Xia \cite{baixia2026alexandrov}, together with Florian John and Lauro Silini \cite{johnsilini2026}, have independently and simultaneously proved an Alexandrov-type characterization: any closed embedded \(\lambda\)-self-expander in Euclidean space must be a round sphere centered at the origin.

Curvature pinching provides another natural approach to rigidity problems for $\lambda$-self-expanders. Ancari and Cheng \cite{AncariCheng} obtained a gap theorem for closed $\lambda$-self-expanders under a pointwise curvature bound involving the second fundamental form.

Motivated by the above results, we study spherical rigidity for closed immersed $\lambda$-self-expanders under two different curvature conditions: the constancy of the squared norm of the second fundamental form and a scalar curvature pinching condition.
We establish the following two rigidity theorems.

\begin{theorem}\label{thm1}
Let $x:M^n\to \mathbb{R}^{n+1}$ be a closed immersed $\lambda$-self-expander. If the squared norm of
the second fundamental form is constant, then $M^n$ is a round sphere centered at the origin.
\end{theorem}

\begin{theorem}\label{thm2}
Let $x:M^n\to \mathbb{R}^{n+1}$ be a closed immersed $\lambda$-self-expander. If the scalar curvature $$R\geq \frac{n-2}{n-1}H^2> 0,$$ then $R=\frac{n-1}{n}H^2$ and $M^n$ is a round sphere centered at the origin.
\end{theorem}

Using the Gauss equation,  $R=H^2-S$, then Theorem \ref{thm2} immediately yields the following result.
\begin{corollary}\label{co-1}
Let $x:M^n\to \mathbb{R}^{n+1}$ be a closed immersed $\lambda$-self-expander. If the mean curvature $H> 0$ and the squared norm of
the second fundamental form $S\leq \frac{H^2}{n-1}$,  then $M^n$ is a round sphere centered at the origin.
\end{corollary}

The organization of this paper is as follows.  In section 2, we  recall some basic notation about $\lambda$-self-expanders. In section 3,   we prove Theorem \ref{thm1} and Theorem \ref{thm2}.

\
\section{Preliminaries}
In this section, we give the necessary facts about  the $\lambda$-self expanders in $\mathbb{R}^{n+1}$.
Let $x: M^n\to \mathbb{R}^{n+1}$ be an $n$-dimensional immersed  hypersurface in the  Euclidean space $\mathbb{R}^{n+1}$. Let $\{e_1,\cdots,e_n\}$ be a local orthonormal frame field of $x(M^n)$ with respect to the first fundamental form $I=\langle dx,dx\rangle$ with dual coframe field $\{\omega_1,\cdots,\omega_n\}$. Let $\xi$ be the unit normal vector field of $x$.
The tangential component of the position vector is
$ x^\top=\sum_{i=1}^n \langle x,e_i\rangle\, e_i, $
and the normal component is
$ x^\perp = \langle x,\xi\rangle\,\xi,$
so that we have
\begin{equation}\label{po-1}
x=x^{\top}+x^\perp, ~~|x|^2=|x^{\top}|^2+\langle x,\xi\rangle^2.
\end{equation}

In this paper we use the following conventions on the ranges of indices:
$$ 1\leq i,j,k,l\leq n.$$
Thus we have the structure equation of $x$,
\begin{equation}\label{eq-str}
\begin{split}
&dx=\sum_{i} \omega_{i}e_{i},\\
&de_{i}=\sum_{j} \omega_{ij}e_{j}+\sum_jh_{ij}\omega_j\xi,\\
&d\xi=-\sum_{i,j}h_{ij}\omega_je_{i},
\end{split}
\end{equation}
where $h_{ij}$ denote the components of the second fundamental form of $x$
 and $\omega_{ij}$ are the Levi-Civita connection forms of $I$ with respect to $\{\omega_1,\cdots,\omega_n\}$.

Let
$$ h=\sum_{i,j} h_{ij}\omega_{i}\otimes \omega_{j} $$
be the second fundamental form, then the squared norm of the second fundamental form $S=|h|^2=tr(h^2)=\sum_{i,j}(h_{ij})^{2}, $  and the mean curvature $H=tr(h)=\sum_{i} h_{ii}.$
The traceless second fundamental form is
$$B=h-\frac{H}{n}I.$$
Then
\begin{equation}\label{eq-0}
\begin{split}
&tr(B)=0, ~~S=|B|^2+\frac{H^2}{n},\\
&tr(h^3)=tr(B^3)+\frac{3H}{n}|B|^2+\frac{H^3}{n^2}.
\end{split}
\end{equation}

The Gauss equations and the Codazzi equations of the hypersurface $x$ are given by
\begin{equation}\label{eq-1}
\begin{split}
&R_{ijkl}=h_{ik}h_{jl}-h_{il}h_{jk},~~1\leq i,j,k,l\leq n,\\
&h_{ij,k}=h_{ik,j}, ~~1\leq i,j,k\leq n,
\end{split}
\end{equation}
where the covariant derivative of $h_{ij}$ is defined by
$$\sum_kh_{ij,k}\omega_k=dh_{ij}+\sum_kh_{ik}\omega_{kj}+\sum_kh_{kj}\omega_{ki}.$$

The second covariant derivative of the second fundamental form is defined by
$$\sum_mh_{ij,km}w_{m}=dh_{ij,k}+\sum_mh_{mj,k}w_{mi}+\sum_mh_{im,k}w_{mj}+\sum_mh_{ij,m}\omega_{mk}.$$
Thus we have the following Ricci identity
\begin{equation}\label{ricd}
h_{ij,kl}-h_{ij,lk}=\sum_mh_{mj}R_{mikl}+\sum_mh_{im}R_{mjkl}.
\end{equation}

By the Gauss equation (\ref{eq-1}), we obtain the Ricci curvature $R_{ij}$ and the scalar curvature $R$ of the hypersurface,
\begin{equation}\label{ricc}
\begin{split}
&R_{ij}=Hh_{ij}-\sum_{m}h_{im}h_{mj},\\
&R=H^2-S,
\end{split}
\end{equation}

Let $f$ be a smooth function on the hypersurface $x(M^n)$. Define the weighted mean curvature $H_f$ of $x$ of weight $e^{-f}$ by
$$H_f:=H-\langle \nabla f, \xi\rangle.$$
The hypersurface $x$ is called constant weighted mean curvature hypersurface (of weight $e^{-f}$) if $H_f=\lambda$ for some constant $\lambda$, i.e.,
$$H=\langle \nabla f, \xi\rangle +\lambda.$$
When $\lambda=0$, the hypersurface $x$ is called $f$-minimal hypersurface. As the models of singularity for the mean curvature flow, self-shrinkers, self-expanders and translating solitons are three special types of $f$-minimal hypersurface, whose weight function $f$ is corresponding to $\frac{|x|^2}{4}, \frac{-|x|^2}{4}$ and $-\langle x, a\rangle$ respectively, where $a\in \mathbb{R}^{n+1}$ is a constant vector.

The constant weighted mean curvature hypersurface with weight $\frac{-|x|^2}{4}$ is called a $\lambda$-self-expander. Thus an immersed hypersurface $x: M^n \to \mathbb{R}^{n+1}$ is called a $\lambda$-self-expander if it satisfies
\begin{equation}\label{def-2}
    H - \frac12 \langle x,\xi \rangle = \lambda,
\end{equation}
where $\xi$ is  the inward unit normal vector field of $x$.

Differentiating  (\ref{def-2}), we have the following equation,
\begin{equation}\label{self-1}
\begin{split}
&H_i=\frac{-1}{2}\sum_kh_{ik}\langle x,e_k\rangle,~~1\leq i\leq n,\\
&H_{ij}=\frac{-1}{2}\sum_kh_{ik,j}\langle x,e_k\rangle-\frac{1}{2}h_{ij}-\sum_kh_{ik}h_{kj}(H-\lambda).
\end{split}
\end{equation}

On $\lambda$-self-expander, the corresponding drifted Laplacian is
$$L = \Delta +\frac{1}{2}\nabla_{x^T}.$$
where \(\Delta\) denotes the Laplacian operator and \(x^T\) is the tangential component  of the position vector \(x\). The second order  elliptic operator $L$ is a self-adjoint operator in $L^2(M^n, e^{\frac{|x|^2}{4}}dv)$, for $f,g\in C^{\infty}(M^n)$,
$$\int_{M^n}f L ge^{\frac{|x|^2}{4}}dv=\int_{M^n}g L fe^{\frac{|x|^2}{4}}dv.$$
Moreover,
$$\int_{M^n}L fe^{\frac{|x|^2}{4}}dv=0.$$
The following equation on the $\lambda$-self-expander is key in the proof of our results.
\begin{lemma}\label{le-4-1}
Let $x:M^n\to \mathbb{R}^{n+1}$ be a $\lambda$-self-expander, then we have
\begin{equation}\label{eq2-1}
\begin{split}
&LH=S(\lambda-H)-\frac{1}{2}H,\\
&LS=2|\nabla h|^2-S-2S^2+2\lambda Str(h^3),\\
&L|x|^2=|x|^2+2\lambda\langle x,\xi\rangle+2n.
\end{split}
\end{equation}
\end{lemma}
\begin{proof}
From (\ref{self-1}),
$$H_{ij}=\frac{-1}{2}\sum_kh_{ik,j}\langle x,e_k\rangle-\frac{1}{2}h_{ij}-\sum_kh_{ik}h_{kj}(H-\lambda).$$
Taking the trace and using the Codazzi equation,
$$\sum_ih_{ij,i}=H_j,$$
we obtain
$$\triangle H=-\frac{1}{2}\langle x^{\top},\nabla H\rangle-\frac{1}{2}H-(H-\lambda)S.$$
Adding the drift term gives the first equation in (\ref{eq2-1}).

From (\ref{eq-1}) and (\ref{ricd}), we can obtain the Simons identity
$$\Delta h_{ij}=\sum_kh_{ij,kk}=H_{ij}+H\sum_kh_{ik}h_{kj}-Sh_{ij}.$$
Substituting the formula for $H_{ij}$, we obtain
$$\Delta h_{ij}=-\frac{1}{2}\sum_kh_{ik,j}\langle x, e_k\rangle-(S+\frac{1}{2})h_{ij}+\lambda\sum_kh_{ik}h_{kj}.$$
Adding the drift term gives the second equation in (\ref{eq2-1}).

From (\ref{eq-str}), $$\Delta x=H \xi.$$ Thus
$$\Delta |x|^2=2H\langle x,\xi\rangle+2n.$$
Since $\nabla|x|^2=2x^{\top}$, then
$$\frac{1}{2}\langle x^{\top}, \nabla |x|^2\rangle=|x^{\top}|^2.$$
Using (\ref{def}) and (\ref{po-1}), we obtain the third equation in (\ref{eq2-1}).
\end{proof}
\begin{lemma}\label{le-4-2}
Let $x:M^n\to \mathbb{R}^{n+1}$ be a $\lambda$-self-expander. If $H>0$, then we have
\begin{equation}\label{eq2-2}
\begin{split}
L\Big(\frac{S}{H^2}\Big)&=-\frac{2}{H}\langle \nabla\Big(\frac{S}{H^2}\Big), \nabla H\rangle+\frac{2}{H^4}\sum_{ij,k}|Hh_{ij,k}-H_kh_{ij}|^2\\
&+\frac{2\lambda}{H^3}\Big[Htr(h^3)-S^2\Big].
\end{split}
\end{equation}
\end{lemma}
\begin{proof}
If $f,g\in C^{\infty}(M^n)$ and $g>0$, then
\begin{equation}\label{le2-1}
L\Big(\frac{f}{g}\Big)=\frac{gL f-fL g}{g^2}-\frac{2}{g}\langle \nabla\Big(\frac{f}{g}\Big), \nabla g\rangle.
\end{equation}
From (\ref{eq2-1}),
$$L H^2=-(S+\frac{1}{2})H^2+2\lambda SH+2|\nabla H|^2.$$
Combining the second formula in (\ref{eq2-1}) and (\ref{le2-1}) proves (\ref{eq2-2}).
\end{proof}

It is well known that there exist no closed self-expanders in $\mathbb{R}^{n+1}$. But the situation is different for $\lambda\neq 0$. In \cite{AncariCheng},
S. Ancari and X. Cheng proved the following results,
\begin{theorem}\cite{AncariCheng}\label{le2-2}
Let $x:M^n\to R^{n+1}$ be a smooth immersed closed $\lambda$-self-expander. Then,\\
(i)~ $\lambda\geq \sqrt{2n}$. Moreover, the sphere $\mathbb{S}^n(\sqrt{2n})$ is the only closed $\lambda$-self-expander with $\lambda=\sqrt{2n}$.\\
(ii)~ If $$S\leq -\frac{1}{2}+\frac{\lambda(\lambda-\sqrt{\lambda^2-2n})}{2n},$$ then $x$ is a sphere $\mathbb{S}^n(r)$ with $r=\lambda+\sqrt{\lambda^2-2n}$.
\end{theorem}

\section{Proof of Main theorems}

{\bf Proof of Theorem \ref{thm1}}.  For $f\in C^{\infty}(M^n)$, define the weighted mean of the function,
$$\bar{f}=\frac{\int_{M^n}fe^{\frac{|x|^2}{4}}dv}{\int_{M^n}e^{\frac{|x|^2}{4}}dv}.$$
Using the linearity of the weighted mean and the weighted Cauchy-Schwarz inequality, we obtain
\begin{equation}\label{wm-1}
\overline{f^2}-(\bar{f})^2=\overline{(f-\bar{f})^2}\geq 0.
\end{equation}

Since $S$ is constant, the first equation in (\ref{eq2-1}) gives
\begin{equation}\label{wm-2}
\overline{H}=\frac{\lambda S}{S+\frac{1}{2}}.
\end{equation}
Combining (\ref{wm-1}) and  $nS\geq H^2$,
$$nS\geq \overline{H^2}\geq (\overline{H})^2=\Big(\frac{\lambda S}{S+\frac{1}{2}}\Big)^2.$$
Thus
\begin{equation}\label{wm-3}
n\Big(S+\frac{1}{2}\Big)^2\geq\lambda^2S.
\end{equation}

Combining (\ref{def}) and (\ref{wm-2}),
\begin{equation}\label{wm-4}
\overline{\langle x,\xi\rangle}=2(\bar{H}-\lambda)=\frac{-\lambda }{S+\frac{1}{2}}.
\end{equation}
Combining the third equation in (\ref{eq2-1}),
\begin{equation}\label{wm-5}
\overline{|x|^2}=\frac{2\lambda^2}{S+\frac{1}{2}}-2n.
\end{equation}
Since $|x|^2=|x^{\top}|^2+\langle x,\xi\rangle^2$, thus
$$\overline{|x|^2}-\Big(\overline{\langle x,\xi\rangle}\Big)^2=\overline{|x^{\top}|^2}
+\overline{\Big(\langle x,\xi\rangle-\overline{\langle x,\xi\rangle}\Big)^2}\geq 0.$$
Combining (\ref{wm-4}) and (\ref{wm-5}), we obtain
$$\overline{|x|^2}-\Big(\overline{\langle x,\xi\rangle}\Big)^2=\frac{2\Big(\lambda^2S-n(S+\frac{1}{2})^2\Big)}{(S+\frac{1}{2})^2}.$$
Thus
$$\lambda^2S\geq n(S+\frac{1}{2})^2.$$
Combining (\ref{wm-3}), we have
$$\lambda^2S=n(S+\frac{1}{2})^2.$$
Hence
$$x^{\top}\equiv 0,~~\langle x,\xi\rangle\equiv \overline{\langle x,\xi\rangle}=constant.$$
Thus
$$|x|^2=|\overline{\langle x,\xi\rangle}|^2=constant.$$
So $M^n$ is a sphere centered at the origin and Theorem \ref{thm1} is proved.

{\bf Proof of Theorem \ref{thm2}}. The proof needs the following Okumura's inequality,
\begin{lemma}\cite{okumura1974hypersurfaces}
Let $\lambda_1, \lambda_2, \cdots, \lambda_n, n\geq 3$, be real numbers such that $\sum_i\lambda_i=0$ and $\sum_i\lambda_i^2=\mu^2$. Then
$$\frac{-(n-2)\mu^3}{\sqrt{n(n-1)}}\leq \sum_i\lambda_i^3\leq \frac{(n-2)\mu^3}{\sqrt{n(n-1)}},$$
and equality holds if and only if at least $(n-1)$ of $\lambda_i$ are equal.
\end{lemma}

Since $R=H^2-S$, the assumption
\[
R\geq \frac{n-2}{n-1}H^2>0
\]
implies
\[
H\neq0,\qquad S\leq \frac{H^2}{n-1}.
\]
Since $M^n$ is connected, $H$ has a fixed sign on $M^n$. Moreover, by Lemma \ref{le2-2}, $$\lambda\geq\sqrt{2n}>0.$$ Suppose that $H<0$. Then, using
\[
LH=S(\lambda-H)-\frac12H,
\]
we have
\[
LH>0
\]
on $M^n$, since $S\geq0$, $\lambda-H>0$, and $-H/2>0$. On the other hand, the self?adjointness of $L$ yields
\[
\int_{M^n} LH\, e^{|x|^2/4}\,dv=0,
\]
which is a contradiction. Hence $H>0$. Therefore,
\begin{equation}\label{th2-1}
H>0,\qquad S\leq\frac{H^2}{n-1}.
\end{equation}
Using (\ref{eq-0}) and the Okumura's inequality, we have
\begin{equation}\label{thm2-1}
\begin{split}
&Htr(h^3)-S^2=Htr(B^3)+\frac{H^2}{n}|B|^2-|B|^4\\
&\geq -|B|^2\Big(|B|^2+\frac{n-2}{\sqrt{n(n-1)}}\frac{H}{n}|B|-\frac{H^2}{n}\Big)\\
&=-|B|^2\Big(|B|-\frac{H}{\sqrt{n(n-1)}}\Big)\Big(|B|+\sqrt{n(n-1)}\frac{H}{n}\Big).
\end{split}
\end{equation}
Since $H>0$ and $\lambda>0$, combining (\ref{eq2-2}) and (\ref{thm2-1}), we obtain
\begin{equation}\label{thm2-2}
\begin{split}
L\Big(\frac{S}{H^2}\Big)&\geq-\frac{2}{H}\langle \nabla\Big(\frac{S}{H^2}\Big), \nabla H\rangle+\frac{2}{H^4}\sum_{ij,k}|Hh_{ij,k}-H_kh_{ij}|^2\\
&-\frac{2\lambda}{H^3}|B|^2\Big(|B|-\frac{H}{\sqrt{n(n-1)}}\Big)\Big(|B|+\sqrt{n(n-1)}\frac{H}{n}\Big).
\end{split}
\end{equation}
When $S\leq \frac{H^2}{n-1}$, then $|B|^2\leq \frac{H^2}{n(n-1)}$ and
\begin{equation}\label{thm2-3}
L\Big(\frac{S}{H^2}\Big)\geq 0.
\end{equation}
Since $M^n$ is closed, the function $\frac{S}{H^2}$ attains its maximum at some point $p\in M^n$, then
$$\frac{S}{H^2}(p)>0, ~~\nabla\Big(\frac{S}{H^2}\Big)(p)=0,~~L\Big(\frac{S}{H^2}\Big)(p)=\Delta\Big(\frac{S}{H^2}\Big)(p)\leq 0,$$
which is a contradiction. Hence  $L\Big(\frac{S}{H^2}\Big)\equiv0$ and $|B|^2=0$. Then $M^n$ is totally umbilical. Since $M^n$ is closed, it is a sphere centered at the origin and Theorem \ref{thm2} is proved.\\

{\bf Acknowledgements:} The first author is supported by the grant No. 12071028 of NSFC.

\end{document}